\RequirePackage{fix-cm}
\documentclass[smallextended]{svjour3-b}       % onecolumn (second format)
\smartqed  % flush right qed marks, e.g. at end of proof
\usepackage{graphicx}
\usepackage{amssymb,xy}
\newtheorem{algorithm}[theorem]{Algorithm}

\newcommand{\Z}{\mathbb{Z}}

\newcommand{\D}{\displaystyle}

\newcommand{\Boxe}{\hfill\rule{2.2mm}{2.2mm}}

\begin{document}

\title{Searching for partial Hadamard matrices}

%\thanks{All authors are partially supported by the research projects FQM--296 and P07-FQM-02980 from Junta de Andaluc\'{\i}a and MTM2008-06578 from Ministerio de Ciencia e Innovaci\'{o}n (Spain).}

%\title{Insert your title here%\thanks{Grants or other notes
%about the article that should go on the front page should be
%placed here. General acknowledgments should be placed at the end of the article.}
%}
%\subtitle{Do you have a subtitle?\\ If so, write it here}

%\titlerunning{Short form of title}        % if too long for running head

\author{V\'{\i}ctor \'{A}lvarez%
\and Jos\'{e} Andr\'{e}s Armario\and Mar\'{\i}a Dolores Frau\and F\'{e}lix Gudiel\and Mar\'{\i}a Bel\'{e}n G\"{u}emes \and Elena Mart\'{\i}n \and Amparo Osuna
}

%\author{First Author         \and
%        Second Author %etc.
%}

\authorrunning{V.Alvarez, J.A.Armario, M.D.Frau, F.Gudiel, M.B.G\"{u}emes, E.Mart\'{\i}n, A.Osuna} % if too long for running head

\institute{V. \'{A}lvarez, J.A. Armario, M.D. Frau, F. Gudiel, E. Mart\'{\i}n, A. Osuna \at
              Dept. Matem\'{a}tica Aplicada 1, ETSII, Avda. Reina Mercedes s/n, 41012, Sevilla, Spain \\
              Tel.: +34-95455-2797, -4386, -4389, -6225, -2798, -2798\\
              Fax: +34-954557878\\
              \email{\{valvarez,armario,mdfrau,gudiel,emartin,aosuna\}@us.es}           %  \\
%             \emph{Present address:} of F. Author  %  if needed
%           \and
%J.A. Armario \at
%              Dept. Matem\'{a}tica Aplicada 1, ETSII, Avda. Reina Mercedes s/n, 41012, Sevilla, Spain \\
%              Tel.: +34-954554386\\
%              Fax: +34-954557878\\
%              \email{armario@us.es}
%                         \and
%M.D. Frau \at
%              Dept. Matem\'{a}tica Aplicada 1, ETSII, Avda. Reina Mercedes s/n, 41012, Sevilla, Spain \\
%              Tel.: +34-954554389\\
%              Fax: +34-954557878\\
%              \email{mdfrau@us.es}           \and
%F. Gudiel \at
%              Dept. Matem\'{a}tica Aplicada 1, ETSII, Avda. Reina Mercedes s/n, 41012, Sevilla, Spain \\
%              Tel.: +34-954556225\\
%              Fax: +34-954557878\\
%              \email{valvarez@us.es}           \and
\and M.B. G\"{u}emes \at
              Dept. Algebra, Fac. Matem\'{a}ticas, Avda. Reina Mercedes s/n, 41012, Sevilla, Spain \\
              Tel.: +34-954556969\\
              Fax: +34-954556938\\
              \email{bguemes@us.es}           \and
%E. Mart\'{\i}n \at
%              Dept. Matem\'{a}tica Aplicada 1, ETSII, Avda. Reina Mercedes s/n, 41012, Sevilla, Spain \\
%              Tel.: +34-954552798\\
%              Fax: +34-954557878\\
%              \email{valvarez@us.es}           \and
%A. Osuna \at
%              Dept. Matem\'{a}tica Aplicada 1, ETSII, Avda. Reina Mercedes s/n, 41012, Sevilla, Spain \\
%              Tel.: +34-954552798\\
%              Fax: +34-954557878\\
%              \email{valvarez@us.es}
                         }

%\date{Received: date / Accepted: date}
% The correct dates will be entered by the editor

\maketitle

%\begin{center}

\noindent {\em In honour of Kathy Horadam.}

\noindent {\em In Memoriam:} This paper is dedicated to the late Warwick Richard de Launey (Oct. 1, 1958 to Nov. 8, 2010), for his outstanding contributions in Design Theory and related topics.
%\end{center}

\begin{abstract}
Three algorithms looking for pretty large partial Hadamard matrices are described. Here ``large'' means that hopefully about a third of a
Hadamard matrix (which is the best asymptotic result known so far, \cite{dLa00}) is achieved. The first one performs some kind of local exhaustive search, and consequently is expensive from the time consuming point of view. The second one comes from the adaptation of the best genetic algorithm known so far searching for cliques in a graph, due to Singh and Gupta \cite{SG06}. The last one consists in another heuristic search,
which prioritizes the required processing time better than the final size of the partial Hadamard matrix to be obtained.
 In all cases, the key idea is characterizing the adjacency
properties of vertices in a particular subgraph $G_{t}$ of Ito's Hadamard Graph $\Delta (4t)$ \cite{Ito85}, since cliques of order $m$ in
$G_{t}$ can be seen as $(m+3) \times 4t$ partial Hadamard matrices.

% \PACS{PACS code1 \and PACS code2 \and more}
%MAth classification: 05B20 (Hadamard matrices), 05C69 (cliques)

\end{abstract}

\section{Introduction}

\label{sec1}

Hadamard matrices consist in $\{1,-1\}$-square matrices whose rows are pairwise orthogonal. This nice property makes Hadamard matrices
being objects for multiple applications (see \cite{HW78} and \cite{Hor06} for instance).

It may be straightforwardly checked that such a matrix must be of size 1, 2 or a multiple of 4. The Hadamard Conjecture claims that a
matrix of this type exists for every size multiple of 4. Many attempts have been devoted to prove this conjecture (both from a
constructive way \cite{Hor06} and also from a theoretical point of view in terms of asymptotic results of existence \cite{dLa09,dLK09}),
but it remains unsolved so far.

From the practical point of view, taking into account possible applications, sometimes there is no need to consider a full Hadamard
matrix. In fact, it suffices to meet a large amount of pairwise orthogonal rows \cite{AAFMO07}. This has originated the interest in
constructing {\em partial Hadamard matrices} $PH$, that is, $m \times n$ $(1,-1)$-matrices $PH$ satisfying $PH \cdot PH^T = nI_m$, for $m
\leq n$. We call $m$ the {\em depth} of $PH$.

From the orthogonality law, it is readily checked that the number $n$ of columns must be 1, 2 or a multiple of 4.

Notice that a partial Hadamard matrix does not need to be a submatrix of a proper Hadamard matrix (for instance, cliques listed in Table \ref{maxcliques}
for $7\leq t \leq 10$ are maximal but not maximum; in the sense that although no larger cliques exist containing them, there exist larger cliques, for example, those related to full Hadamard matrices).

%{\sc Incluir un ejemplo con el grafo $G_{t}$.}

Although partial Hadamard matrices are as useful as Hadamard matrices themselves with regards to practical purposes, unfortunately it seems
that their explicit construction is equally hard as well.

De Launey proved in \cite{dLa00} that partial Hadamard matrices of size about a third of a $4t \times 4t$ Hadamard matrix  exist for large
$t$. The proof gives a polynomial time algorithm in $t$ for constructing such a matrix. Furthermore, De Launey and Gordon proved in \cite{dLG01} that about a half of a Hadamard matrix $4t \times 4t$ exists for large $t$, assuming that the
Riemann hypothesis is true. The idea is decomposing $2t-i$ as the sum of $i$ odd prime numbers $p_i$, $2 \leq i \leq 3$, so that the juxtaposition of the corresponding Paley conference matrices provides a partial Hadamard matrix of depth $2 \min \{ p_i\} +2$.
Unfortunately, none of these methods can provide a partial Hadamard matrix of depth greater than half of a full Hadamard matrix.

%What can be said about partial Hadamard matrices then?

In this paper, we present three new algorithms for constructing partial Hadamard matrices of size $m \times 4t$. The first one performs some kind of local exhaustive search, and consequently is expensive from the time consuming point of view. The second one comes from the adaptation of the best genetic algorithm known so far searching for cliques in a graph, due to Singh and Gupta \cite{SG06}. The third one consists in another heuristic search, which prioritizes the required processing time better than the final size of the partial Hadamard matrix obtained so far. The idea is looking for large cliques (i.e. subgraphs whose vertices are pairwise adjacent) in a subgraph $G_{t}$ of Ito's Hadamard Graph $\Delta (4t)$
\cite{Ito85}.

Although the results showed in \cite{dLa00} and \cite{dLG01} are impressive and meaningfully better than any obtained from the algorithms described in this paper, it is a remarkable fact that our algorithms may provide partial Hadamard matrices of depth greater than half of a full Hadamard matrix. Furthermore, it is possible (and desirable) to run our algorithms taking as input data partial Hadamard matrices obtained by the procedures in \cite{dLa00,dLG01}, so that deeper partial Hadamard matrices are constructed (see Table \ref{delauneymejorado}).

%Experimental results show that very often $\D m\geq \frac{t}{3}$, which is the bound reached by the polynomial time algorithm described in \cite{dLa00}.

%As soon as something new is known about the adjacency relation between vertices in $G_t$, the algorithm could be straightforwardly
%improved.

%Our search procedure is based on the best genetic algorithm for the Maximum Clique Problem \cite{SG06}  known so far. We have designed a
%new heuristic for the extending function of this procedure, attending to the particular properties of the graph $G_{4t}$, which
%substantially improves the search. In fact, there is no need to apply the full GA, since the only execution of the extending function
%provides a clique of nice size.

The paper is organized as follows.

The graph $G_{t}$ and its properties are described in Section 2. Section 3 is devoted to the description of the local exhaustive algorithm looking for
cliques in $G_{t}$. In Section 4, the two heuristics searching for cliques in $G_t$ are described.
%Section 2 is devoted to the description of Singh and Gupta's algorithm \cite{SG06} for the Maximum Clique Problem. The new heuristic for extending small cliques is described in Section 3.
Last section is
devoted to conclusions.

\section{The graph $G_{t}$}

\label{sec2}

In what follows, for clarity in the exposition, we will simply use $+$ and $-$ instead of $1$ and $-1$.

{\em Hadamard Graphs} were introduced by Ito in \cite{Ito85}. Originally they referred to the graph $\Delta(4t)$ whose vertices are the
$(1,-1)$-vectors of length $4t$ consisting of an even number of $1$s. The adjacency relation consists in orthogonality.

We call Hadamard graph to the subgraph $G_t$ of $\Delta(4t)$ induced by the $(1,-1)$-vectors simultaneously orthogonal to the three first
rows of a normalized Hada\-mard matrix,
$$\left(\begin{array}{ccc|ccc|ccc|ccc} + &\ldots & + & + &\ldots & +& + & \ldots  & + &  + &\ldots &+\\
+& \ldots & +&  + & \ldots & + & - & \ldots & - & - &  \ldots & -  \\
+& \ldots & + & - & \ldots & - & + & \ldots & + & - & \ldots & - \\
& \ldots &  & &\ldots & & &\ldots && &\ldots &  \end{array}\right)$$

These orthogonality conditions straightforwardly characterize the form of the vertices in $G_t$.

\begin{lemma} The vertices of $G_t$ consist of $(1,-1)$-vectors of length $4t$ where the $2t$ negative entries are distributed following
this pattern, \begin{center} \includegraphics[width=80mm]{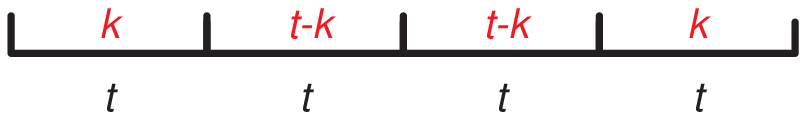} \end{center}

\noindent so that exactly $k$, $t-k$, $t-k$ and $k$ negative entries occur among every $t$ positions, for some $0 \leq k \leq t$.

\end{lemma}

%\begin{proof}

%It may be readily checked, taking into account the orthogonality

%Orthogonality to the first normalized row implies that exactly half the entries have to be positive and the other half must be negative,
%\begin{center} \includegraphics[width=80mm]{dist2b.eps} \end{center}

%If additionally one assumes orthogonality to the second normalized row, it is readily checked that these $2t$ negative entries are
%homogeneously distributed over columns 1 to $2t$ and $2t+1$ to $4t$, so that precisely $t$ -1s occur at each side. \begin{center}
%\includegraphics[width=80mm]{dist3b.eps} \end{center}

%Finally, if orthogonality to the third normalized row is also contemplated, the negative entries are symmetrically distributed by
%quarters, so that exactly $k$, $t-k$, $t-k$ and $k$ positive entries occur at each quarter, for some $0 \leq k \leq t$.
%\begin{center}\includegraphics[width=80mm]{distb.eps} \end{center}

%\Boxe

%\end{proof}

%For clarity in the exposition, from now on we adopt the additive notation, so that positive entries will be represented by $0$, whereas
%negative entries will be codified simply as $1$. Consequently, $G_t$ consists now in $(0,1)$-vectors of length $4t$ distributed in the way
%illustrated above.

We may then classify the set of vertices in $G_t$ attending to the number $k$ of negative entries which appear in positions 1 through $t$.
In what follows, a $k$-vertex in $G_t$ refers to a vertex with precisely $k$ negative entries among positions 1 to $t$.

\begin{lemma} In particular, the number of vertices in $G_t$ is $|G_t|=\displaystyle \sum _{k=0}^t \left( \begin{array}{c}t\\k \end{array}
\right)^4$. \end{lemma}

%It is readily checked that every clique of size $m$ (i.e. subgraph consisting in $m$ pairwise adjacent vertices) in $G_t$ straightforwardly
%translates to a $(m+3) \times 4t$ partial Hadamard matrix.
%Partial Hadamard matrices $m \times 4t$ translate to cliques of size $m$ in $G_t$, that is, subgraphs of $m$ pairwise adjacent vertices of
%$G_t$.

%Unfortunately, finding the maximum clique of a graph is a NP-Hard problem (

%Unfortunately, the size of $G_t$ seems to grow exponentially on $t$, as the table below suggests, so the problem of looking for large
%cliques in $G_t$ might be difficult. In the following table, as usual, $|V|$ denotes the amount of vertices in $G_t$, and $|E|$ denotes
%the amount of edges in $G_t$):
%$$\begin{array}{|c||c|c|c|c|c|c|c|} \hline t& 1&2&3&4&5&6&7\\ \hline \hline |V|& 2&18&164&1810&21252&263844&3395016\\
%\hline |E| & 0&80&5184 &587088&73440000&10521080000&1629606720000\\ \hline \end{array}$$

The tables below give the number of $k$-vertices ($n.v.$) and their degree $\delta$ (number of adjacent vertices), for $1 \leq t \leq 7$
and $0 \leq k \leq t$. The column $total$ refers to the number of vertices and edges in $G_t$ (notice that the number of edges is half the
summation of the degree of every vertex). It is evident that for a fixed value of $k$, every $k$-vertex has the same degree (since
permuting some columns does not affect to the
orthogonality relation).

\begin{center}\begin{table*}[ht]\caption{Vertices and edges in $G_t$.}$$\begin{array}{|c||c|c||c|} t & \multicolumn{2}{c||}{1}& \\ \hline \hline k & 0 & 1 & total\\ \hline n.v.& 1&1&2\\
\hline \delta & 0 & 0 & 0 \\ \hline \end{array} \quad \begin{array}{|c||c|c|c||c|} t & \multicolumn{3}{c||}{2}& \\ \hline \hline k & 0 & 1 & 2&total\\
\hline n.v.& 1&16&1&18\\
\hline \delta & 16 & 8 & 16& 80 \\ \hline \end{array} \quad \begin{array}{|c||c|c|c|c||c|} t & \multicolumn{4}{c||}{3}& \\ \hline \hline
k & 0 & 1 & 2&3&total\\
\hline n.v.& 1&81&81&1&164\\
\hline \delta & 0 & 64 & 64& 0&5184 \\ \hline \end{array}$$ $$\begin{array}{|c||c|c|c|c|c||c|} t & \multicolumn{5}{c||}{4}& \\ \hline
\hline
k & 0 & 1 & 2&3&4&total\\
\hline n.v.& 1&256&1296&256&1&1810\\
\hline \delta & 1296 & 648 & 648&648& 1296&587088 \\ \hline \end{array} $$ $$\begin{array}{|c||c|c|c|c|c|c||c|} t & \multicolumn{6}{c||}{5}& \\
\hline \hline
k & 0 & 1 & 2&3&4&5&total\\
\hline n.v.& 1&625&10000&10000&625&1&21252\\
\hline \delta & 0&6912 & 6912 & 6912&6912& 0&73440000 \\ \hline \end{array}$$ $$\begin{array}{|c||c|c|c|c|c|c|c||c|} t & \multicolumn{7}{c||}{6}& \\
\hline \hline
k & 0 & 1 & 2&3&4&5&6&total\\
\hline n.v.& 1&1296&50625&160000&50625&1296&1&263844\\
\hline \delta & 160000&80000 &79808 & 79712&79808& 80000&160000&10521080000 \\ \hline \end{array}$$
$$\begin{array}{|c||c|c|c|c|c|c|c|c||c|} t & \multicolumn{8}{c||}{7}& \\
\hline \hline
k & 0 & 1 & 2&3&4&5&6&7&total\\
\hline n.v.& 1&2401&194481&1500625&1500625&194481&2401&1&3395016\\
\hline \delta & 0&960000&960000 &960000 & 960000&960000& 960000&0&1629606720000 \\ \hline \end{array}$$\end{table*}\end{center}

%%%%%%%%%%% INCLUIR TABLAS

It is readily checked that the size of $G_t$ grows exponentially on $t$. Since cliques of size $m$ in $G_t$ translate to partial Hadamard
matrices $(m+3) \times 4t$, we would like to search for large cliques in $G_t$. Since the largest clique in $\Delta (4t)$ is at most of
size $4t$ (see \cite{Ito84} for details), the largest clique in $G_t$ is at most of size $4t-3$. Cliques meeting the upper bound would
correspond to full Hadamard matrices.

%Unfortunately, determining the size of a maximum clique in an arbitrary graph is a well-known NP-complete problem \cite{BBPP99}. Taking
%into account the size of $G_t$, this problem seems to be hard as well.

%\section{The Maximum Clique Problem (MCP)}

\begin{example}

For instance, consider the graph $G_2$, obtained from $t=2$. The picture below shows a clique of size 5, so that adding the three
normalized rows we obtain a full Hadamard matrix of size $8 \times 8$.

$\begin{array}{l} {\bf (--++++--)=1}\\ {\bf (-+-+-+-+)=2}\\ (-+-+-++-)=3\\ (-+-++--+)=4\\ {\bf (-+-++-+-)=5}\\
 (-++--+-+)=6\\ {\bf (-++--++-)=7}\\ {\bf (-++-+--+)=8}\\
(-++-+-+-)=9\\
(+--+-+-+)=10\\ (+--+-++-)=11\\ (+--++--+)=12\\ (+--++-+-)=13\\ (+-+--+-+)=14\\ (+-+--++-)=15\\ (+-+-+--+)=16\\ (+-+-+-+-)=17\\
(++----++)=18 \end{array}$

\vspace*{-7.5cm} \hspace{4cm}\includegraphics[width=75mm]{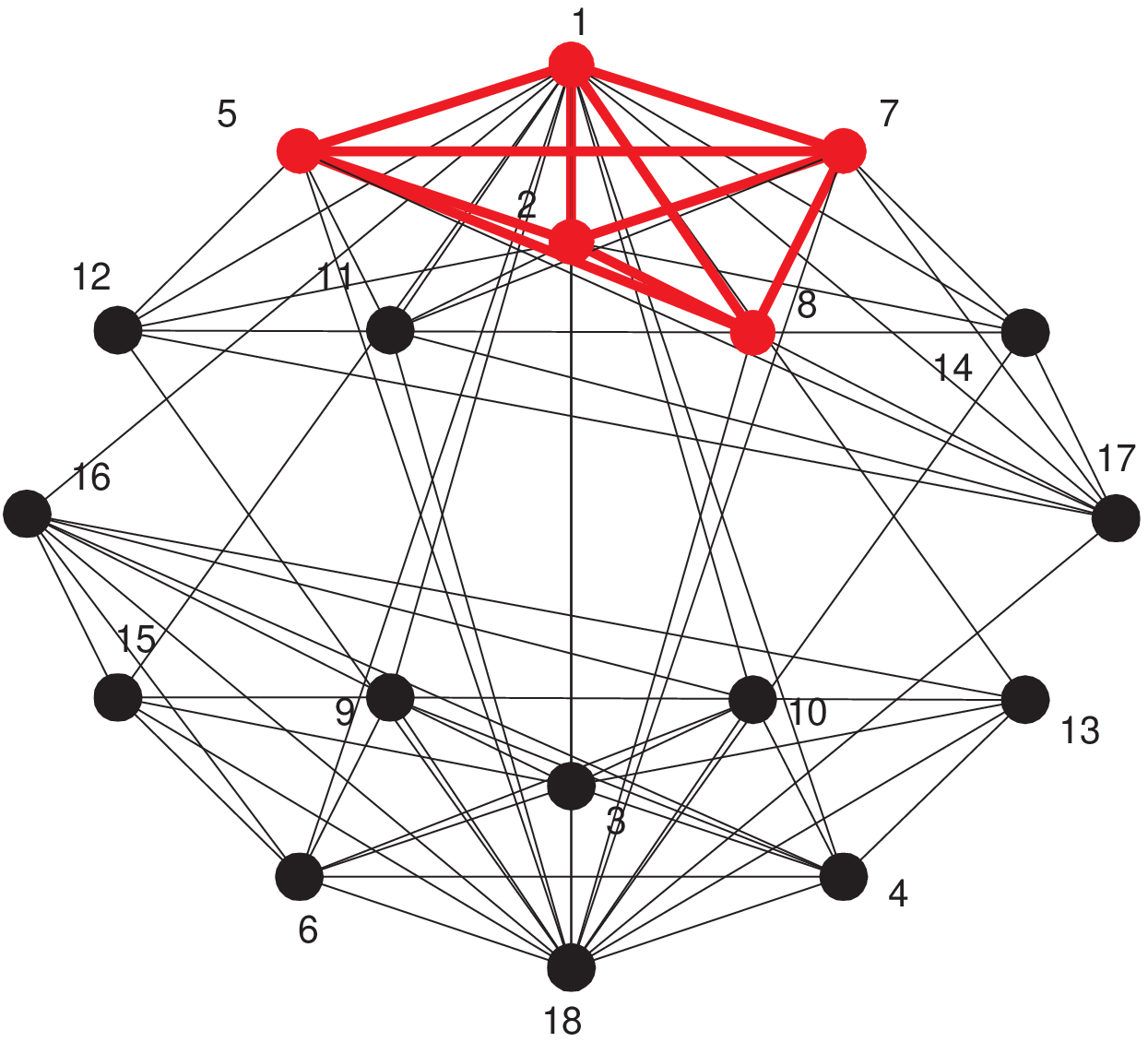}

\hfill{$\Box$}

\end{example}

%{\sc Cambiar los 0 por -}

A {\em maximum clique} is a clique with the maximum cardinality (which is called the maximum clique number). This notion is different from
that of maximal clique, which refers to a clique which is not a proper subset of any other clique. Thus maximal cliques need not be
maximum ones (as we had already noticed in the introduction), though the converse is always true.

%{\sc Citar el ejemplo del principio}

Given a graph, the maximum clique problem (MCP) is to find a maximum clique, and it is NP-complete \cite{BBPP99}. Unfortunately, there is
no polynomial-time algorithm for approximating the maximum clique within a factor of $n ^{1-\epsilon}$ unless P=NP \cite{Has96}, where $n$
is the number of the vertices of the graph. Moreover, there is no polynomial-time algorithm approximating the clique number within a
factor of $\displaystyle \frac{n}{(\log n)^{1-\epsilon}}$ unless NP=ZPP \cite{Kho01}.

Anyway, our purpose here is to design an algorithm for constructing sufficiently large cliques in $G_t$, at least of depth greater than
a third of a full Hadamard matrix. To this end, we need to study the properties of $G_t$ in a more detailed way.

In what follows, for brevity, we will adopt the additive notation for representing Hadamard matrices, so that the $1$s turn to $0$s and the
$-1$s turn to $1$s.

This way, $k$-vectors in $G_t$ are now described as $(0,1)$-vectors of length $4t$ consisting of precisely $2t$ ones (and hence $2t$
zeros), which are distributed in the following way: there are exactly $k$ ones in positions 1 through $t$, other $t-k$ ones in positions
$t+1$ through $2t$, another $t-k$ ones in positions $2t+1$ through $3t$, the last $k$ ones being located in positions $3t+1$ through $4t$.

Each of these $k$-vectors may be straightforwardly codified as an integer, assuming that the $k$-vector is the binary representation of a
decimal number. Therefore cliques are codified as lists of integers, each of them being the decimal representation of a binary number
consisting of $2t$ ones and length less or equal to $4t$.

\begin{lemma} \label{lema1} Actually, it can be assumed that $0\leq k \leq \D \lfloor \frac{t}{2} \rfloor$. \end{lemma}

\begin{proof} This is a straightforward consequence of the fact that the negation of a row does not affect to the set of its orthogonal
vectors.

\Boxe \end{proof}

That being so, the full set of vertices adjacent to a given $k$-vertex {\bf v} may be obtained by calculating those $s$-vectors {\bf w}
orthogonal to {\bf v}, for $0 \leq s \leq \lfloor \frac{t}{2} \rfloor$, and then adding their complements. Notice that with the additive
notation at hand, a $k$-vector {\bf v} and a $s$-vector {\bf w} are orthogonal if and only if they share exactly $2t$ bits.

We now describe a procedure for determining a set $\delta _{\bf v}$ of generators for the adjacency list related to a fixed
$k$-vector ${\bf v}$, that is, generating those $s$-vectors {\bf w} orthogonal to ${\bf v}$.

In order to compute the total amount of coincidences between {\bf v} and {\bf w}, one may split the vectors by quarters, and count the
amount $i$ of $1$-bits (for the first and last quarters) and $0$-bits (for the second and third quarters) that these vectors share in each
of these basic quarters.

\begin{lemma}\label{lema2} At the first and fourth (resp., the second and third) quarters the number $i$ of coincidences in $1$s (resp,
in $0$s) runs in the range $[\max (0,s+k-t), \min (k,s)]$. Actually, assuming the conditions in Lemma \ref{lema1}, $i \in [0,\min (k,s)]$.
\end{lemma}

\begin{lemma}\label{lema3} At each quarter, the number $\alpha _i$ of total coincidences (both in $1$s and $0$s) satisfies $\alpha
_i=t-s-k+2i$, and runs in the range $[|t-k-s|,t-|s-k|]$. Actually, assuming the conditions in Lemma \ref{lema1}, $\alpha_i \in
[t-s-k,t-|s-k|]$.\end{lemma}

\begin{proof}

Assume that at a given quarter the vectors {\bf v} and {\bf w} share exactly $i$ $1$-bits. Then {\bf w} has $k-i$ $0$-bits in those
positions where the remaining $1$-bits of {\bf v} are located. Analogously, {\bf v} has $s-i$ $0$-bits in those positions where the
remaining $1$-bits of {\bf w} are located. Thus {\bf v} and {\bf w} share exactly $t-i-(k-i)-(s-i)=t-k-s+i$ $0$-bits. Since {\bf v} and
{\bf w} share $i$ $1$-bits and $t-k-s+i$ $0$-bits, the total number of coincidences is $\alpha _i=t-k-s+2i$.

\Boxe

\end{proof}

\begin{corollary} \label{coro1}

$\alpha _{i+1}=\alpha _i+2$.

\end{corollary}

Let denote $n=\min (k,s)$. In the conditions above, the set of total coincidences is given by $t-s-k=\alpha _0 < \ldots < \alpha
_n=t-|s-k|$. We may now describe the set of $s$-vectors adjacent to a given $k$-vector.

\begin{proposition}\label{prop2} The set of vectors orthogonal to a given $k$-vector corresponds to the full set of distributions of vectors satisfying
tuples of total coincidences $(\alpha_{i_1},\alpha _{i_2},\alpha _{i_3},\alpha _{i_4})$ such that $\alpha_{i_1}+\alpha _{i_2}+\alpha
_{i_3}+\alpha _{i_4}=2t$. \end{proposition}

\begin{proposition} The set of tuples $(\alpha _{i_1},\alpha _{i_2},\alpha _{i_3},\alpha _{i_4})$ which give rise to orthogonal $s$-vectors
are characterized as the solutions of the following system of diophantine equations
\begin{equation} \label{sis1} \left\{ \begin{array}{ccccccc} x_0 \alpha_0 &+& \ldots& +& x_n \alpha _n &=& 2t\\ x_0&+& \ldots& +&x_n&=&4\\
\multicolumn{7}{c}{x_i \in \Z: \; 0 \leq x_i \leq 4} \end{array} \right. \end{equation} Here, $n=\min (k,s)$ and $x_i$ indicates how many coincidences of the
type $\alpha _i$ must occur among the four quarters. \end{proposition}

We now give a constructive way to solve the system above.

\begin{proposition}\label{prop1}  There exists a solution for the system (\ref{sis1}) iff $4 \alpha _0 \leq 2t \leq 4\alpha _{n}$.
\end{proposition}

\begin{proof}

From Corollary \ref{coro1}, we know that the difference between two consecutive $\alpha _{i_j}$ is 2.

Since $4 \alpha _0$, $2t$ and $4 \alpha _n$ are even and taking into account Corollary \ref{coro1}, we conclude that  every even number in
the range $[4 \alpha _0, 4 \alpha _n]$ may be (not uniquely, in general) written as a combination $\alpha _{i_1}+ \alpha _{i_2}+ \alpha
_{i_3}+ \alpha _{i_4}$ for some values $\mbox{$\alpha _0 \leq \alpha _{i_1} \leq \alpha _{i_2}\leq  \alpha _{i_3}\leq \alpha _{i_4} \leq
\alpha _n$}$.

\Boxe

\end{proof}

The condition above may be straightforwardly generalized for the case of solutions related to tuples of the type $(\alpha _{i_1},\alpha
_{i_2},\alpha _{i_3},\alpha _{i_4})$, $\alpha _{i_1}\leq \alpha _{i_2}\leq \alpha _{i_3}\leq \alpha _{i_4}$.

\begin{corollary} \label{bueno}
Fixed $t$ and $0 \leq k,s \leq \lfloor \frac{t}{2} \rfloor$, the set $sol$ of solutions to the system (\ref{sis1}) may be constructed in the following way:\\
{\tt
$sol\leftarrow  \emptyset$\\
$\alpha _0 \leftarrow t-k-s$\\
$\alpha _n \leftarrow t-|k-s|$\\
for $i_1$ from $\max \{\alpha_1, 2 t - 3 \alpha_n\}$ to $\min\{\alpha_n, \lfloor \frac{2 t}{4}\rfloor \}$ with step 2 do\\
\hspace*{.5cm} for $i_2$ from $\max\{i_1, 2 t - i_1 - 2 \alpha_n\}$ to $\min\{\alpha_n, \lfloor \frac{2 t - i_1}{3}\rfloor \}$ with step 2\\
\hspace*{.5cm}  do\\
\hspace*{1cm} $\mbox{for $i_3$ from $\max\{i_2, 2 t - i_1 - i_2 - \alpha_n\}$ to $\min\{\alpha_n, \lfloor \frac{2 t - i_1 - i_2}{2}\rfloor \}$  with}$\\
 \hspace*{1cm} $\mbox{step 2 do}$\\
\hspace*{1.5cm} $sol\leftarrow sol \cup \{\{i_1,i_2,i_3,2t-i_1-i_2-i_3\}\}$\\
\hspace*{1cm} od\\
\hspace*{.5cm} od\\
od }

\end{corollary}

Given a tuple $(\alpha _{i_1},\alpha _{i_2},\alpha _{i_3},\alpha _{i_4})$ solution to (\ref{sis1}), construct the four matrices $N_k$ whose rows are those vectors satisfying $\alpha _{i_k}$ total coincidences with the corresponding quarter of ${\bf v}$. By construction, the juxtaposition of any of the rows of these matrices gives a vector orthogonal to ${\bf v}$.

\begin{proposition} \label{propelegida}
A set $\delta _{\bf v}$ of generators for the adjacency list of ${\bf v}$ may be straightforwardly constructed in terms of matrices of the type above.
\end{proposition}

In spite of the fact that the size (both in edges and vertices) of
$G_t$ grows exponentially in $t$, the procedure described in Proposition \ref{propelegida} is a cheaper way (in terms of both time and space) of saving this information.

We will illustrate now how the proposition above works by means of an example. In order to simplify the reading, in what follows we will
use a vertical line $|$ to separate the different quarters of a $k$-vector.

\begin{example}

Let us consider the case $t=5$, $k=2$ and ${\bf v}=684646$ (its binary representation gives the $2$-vector ${\bf
v}=(10100|11100|10011|00110)$). We are going to calculate the full set of $s$-vectors orthogonal to {\bf v}, for $\mbox{$0 \leq s \leq
2=\lfloor \frac{5}{2} \rfloor$}$ (recall that the remaining orthogonal vectors are obtained by simply interchanging the $0$-bits and
$1$-bits, since they are the negation of the vectors just calculated).

\begin{enumerate}

\item Case $s=0$.

From Lemma \ref{lema2}, we know that there is just one value for the number $i$ of coincidences in $1$-bits, namely $i=0$. Consequently,
there is just one value $\alpha _i$, namely $\alpha _0=3$.

Since $4 \alpha _0=12 \neq 10=2 \cdot 5$, the system (\ref{sis1}) has no solutions, so there is no $0$-vector orthogonal to {\bf v}.

This was evident from the very beginning, since there is only one $0$-vector, ${\bf w}=32736$ (i.e. ${\bf w}= (00000|11111|11111|00000)$),
and it is not orthogonal to {\bf v}.

\item Case $s=1$.

From Lemma \ref{lema2}, we know that $i \in \{0,1\}$, and hence $\alpha _i \in \{2,4\}$. The set of different ways in which four $\alpha _i$
may be orderless chosen to sum 10 is described by the system (\ref{sis1}),
$$\left\{ \begin{array}{rcrcc} 2x_0&+&4x_1&=&10\\
x_0&+&x_1&=&4 \end{array}\right.$$

Since $4\alpha _0=8 \leq 10 \leq 4 \alpha _1=16$, there exist solutions for the system. In fact, there is just one solution,
$(x_0,x_1)=(3,1)$, which corresponds to the following distribution of total coincidences (up to reordering): $(2,2,2,4)$.

In order to explicitly construct those $1$-vectors {\bf w} meeting the distribution $(2,2,2,4)$, we have to find those $1$-vectors with $0$
coincidences with {\bf v} in $1$-bits in the first quarter, $0$ coincidences in $0$-bits in the second and third quarters, and $1$
coincidence in $1$-bits in the fourth quarter.

Since {\bf v} is a $2$-vector, in the first quarter there are $\left( \begin{array}{c} 5-2 \\ 1 \end{array} \right)$ choices for placing
the $1$-bit of {\bf w} among the $0$-bits of {\bf v}. Analogously, there are $\left( \begin{array}{c} 5-2 \\ 1 \end{array} \right)$
choices for placing the $0$-bit of {\bf w} among the $1$-bits of {\bf v} in the second and third quarters. And there are $\left(
\begin{array}{c} 2 \\ 1 \end{array} \right)$ choices for placing the $1$-bit of {\bf w} among the $1$-bits of {\bf v}.

In conclusion, the set of $1$-vectors meeting the distribution of total coincidences $(2,2,2,4)$ is generated by the juxtaposition of any
of the rows of the following matrices
$$\left( \begin{array}{ccccc} 0&1&0&0&0 \\ 0&0&0&1&0 \\ 0&0&0&0&1 \end{array} \right)  \times \left( \begin{array}{ccccc} 0&1&1&1&1 \\
1&0&1&1&1 \\ 1&1&0&1&1 \end{array} \right)\times \left( \begin{array}{ccccc} 0&1&1&1&1 \\ 1&1&1&0&1 \\ 1&1&1&1&0 \end{array} \right)
\times \left( \begin{array}{ccccc} 0&0&1&0&0 \\ 0&0&0&1&0 \end{array} \right)$$

Similar schemes are achieved with the remaining orderings of the valid distribution of total coincidences, $(2,2,4,2)$, $(2,4,2,2)$ and
$(4,2,2,2)$. In fact, they may be obtained by simply permuting and/or negating some suitable columns of the distribution above.

\item Case $s=2$.

From Lemma \ref{lema2}, we know that $i \in \{0,1,2\}$, and hence $\alpha _i \in \{1,3,5\}$. The set of different ways in which four $\alpha
_i$ may be orderless chosen to sum 10 is described by the system (\ref{sis1}),
$$\left\{ \begin{array}{rcrcrcc} x_0&+&3x_1&+&5x_2&=&10\\
x_0&+&x_1&+&x_2&=&4 \end{array}\right.$$

Since $4\alpha _0=4 \leq 10 \leq 4 \alpha _2=20$, there exist solutions for the system. In fact, there are just two solutions,
$(x_0,x_1,x_2)\in \{(2,1,1),(1,3,0)\}$, which correspond to the following distribution of total coincidences (up to reordering):
$(1,1,3,5)$ and $(1,3,3,3)$.

In order to explicitly construct those $2$-vectors {\bf w} meeting the distribution $(1,1,3,5)$, we have to find those $2$-vectors with $0$
coincidences with {\bf v} in $1$-bits in the first quarter, $0$ coincidences in $0$-bits in the second quarter, $1$ coincidence in
$0$-bits in the third quarter, and $2$ coincidences in $1$-bits in the fourth quarter.

Since {\bf v} is a $2$-vector, in the first quarter there are $\left( \begin{array}{c} 5-2 \\ 2 \end{array} \right)$ choices for placing
the $1$-bits of {\bf w} among the $0$-bits of {\bf v}. Analogously, there are $\left( \begin{array}{c} 5-2 \\ 2 \end{array} \right)$
choices for placing the $0$-bits of {\bf w} among the $1$-bits of {\bf v} in the second quarter. Analogously,  in the third quarter, there
are $\left( \begin{array}{c} 2 \\ 1 \end{array} \right)$ choices for placing one $0$-bit of {\bf w} among the $0$-bits of {\bf v}, and for
each of them, there are $\left( \begin{array}{c} 5-2 \\ 1 \end{array} \right)$ choices for placing the second $0$-bit of {\bf w} among the
$1$-bits of {\bf v}. Finally, there are $\left( \begin{array}{c} 2 \\ 2 \end{array} \right)$ choices (just one!) for placing the $1$-bits
of {\bf w} among the $1$-bits of {\bf v}.

In conclusion, the set of $2$-vectors meeting the distribution of total coincidences $(1,1,3,5)$ is generated by the juxtaposition of any
of the rows of the following matrices $$\left( \begin{array}{ccccc} 0&0&0&1&1 \\ 0&1&0&0&1 \\
0&1&0&1&0 \end{array} \right) \times \left( \begin{array}{ccccc} 1&0&0&1&1 \\ 0&1&0&1&1 \\ 0&0&1&1&1 \end{array} \right)\times \left(
\begin{array}{ccccc} 1&0&1&1&0 \\ 1&0&1&0&1 \\ 0&0&1&1&1\\ 1&1&0&1&0 \\ 1&1&0&0&1 \\ 0&1&0&1&1 \end{array} \right)\times \left(
\begin{array}{ccccc} 0&0&1&1&0 \end{array} \right)$$

Similar schemes are achieved from the remaining orderings of the distribution $(1,1,3,5)$.

Analogously, one may conclude that the set of $2-$vectors meeting the distribution of total coincidences $(1,3,3,3)$ (and in a similar way
its reorderings as well) is generated by the
juxtaposition of any of the rows of the following matrices $$\left( \begin{array}{ccccc} 0&0&0&1&1 \\ 0&1&0&0&1 \\
0&1&0&1&0 \end{array} \right) \times \left( \begin{array}{ccccc} 1&1&0&0&1 \\ 1&0&1&0&1 \\ 0&1&1&0&1  \\ 1&1&0&1&0  \\ 1&0&1&1&0  \\
0&1&1&1&0 \end{array} \right)\times \left( \begin{array}{ccccc} 1&0&1&1&0 \\ 1&0&1&0&1 \\ 0&0&1&1&1\\ 1&1&0&1&0 \\ 1&1&0&0&1 \\ 0&1&0&1&1
\end{array} \right)\times \left( \begin{array}{ccccc}0&0&1&0&1 \\ 0&1&1&0&0 \\ 1&0&1&0&0\\ 0&0&0&1&1 \\ 0&1&0&1&0 \\ 1&0&0&1&0 \end{array}
\right)$$

\end{enumerate}

\hfill{$\Box$}

\end{example}

We have just described a procedure for determining a set $\delta _{\bf v}$ of generators for the adjacency list related to a fixed
$k$-vector ${\bf v}$, $0 \leq k \leq \lfloor \frac{t}{2} \rfloor$. In spite of the fact that the size (both in edges and vertices) of
$G_t$ grows exponentially in $t$, this is a cheaper way (in terms of both time and space) of saving this information.

%\begin{proposition} Given a $k$-vertex ${\bf v}$, the procedure described above for constructing a set $\delta _{\bf v}$ of generators for
%the adjacency list of ${\bf v}$ runs in polynomial time in $t$. \end{proposition}

%\begin{proof}

%On one hand, the complete set of tuples $(\alpha _{i_1},\alpha _{i_2},\alpha _{i_3},\alpha _{i_4})$ which correspond to the solutions to
%the system (\ref{sis1}) may be calculated in linear-time, attending to the condition described in Proposition \ref{prop1}.

%On the other hand, for every solution of (\ref{sis1}), the procedure of

%\Boxe

%\end{proof}

Once we know how the adjacency relation in $G_t$ looks like, we are in conditions to describe an algorithm looking for cliques in $G_t$.

\section{Local exhaustive search}

\label{sec3}

From the results described in the precedent Section, one may straightforwardly design an algorithm searching for a maximal (probably not
maximum) clique in $G_t$. Starting from a clique $C$ initially consisting of a random vertex ${\bf v}$ in $G_t$, it suffices to repeatedly
add a new vertex ${\bf w}$ to $C$, randomly chosen among those vertices simultaneously orthogonal to the vertices already in $C$.

This is somehow a local exhaustive search. Exhaustive, in the sense that repeatedly the full set of vertices simultaneously orthogonal to the nodes of a given clique is constructed. Local, in the sense that just one vertex among the full set of candidates is actually used.

We next include a pseudo-code for this algorithm.

\begin{algorithm} Searching for cliques in $G_t$.

\vspace*{.5cm}

{\tt

\noindent Input: an integer $t$\\
\noindent Output: a maximal clique in $G_t$

\vspace*{.5cm}
%\noindent $\backslash\backslash$ reproduction starts\\

\noindent Select a $k$-vertex ${\bf v}$ in $G_t$\\
\noindent $clique \leftarrow \{{\bf v}\}$\\
\noindent $adj \leftarrow \delta _{\bf v}$\\
\noindent while $adj$ is not empty\{ \\
\hspace*{.5cm} ${\bf v} \leftarrow$ a random $s$-vertex in $adj$\\
\hspace*{.5cm} $clique \leftarrow clique \cup \{ {\bf v}\}$\\
\hspace*{.5cm} $adj \leftarrow adj \cap \delta_{\bf v}$\\
\noindent $\}$\\
\noindent $clique$}

\end{algorithm}

%{\sc Justificar que es de complejidad polin\'{o}mica.}

The complexity of the algorithm relays on the operations $\delta _{\bf v}$ and $adj \cap \delta _{\bf v}$. From Proposition \ref{prop1},
$\delta _{\bf v}$ may be constructed in polynomial-time. Unfortunately, the size of $adj$ increases exponentially, and so the algorithm is
both time and space consuming.

Improved versions of the algorithm might be implemented, depending on a deeper knowledge of how orthogonality on $k$-vertices works.

We now include some execution tables.

All the calculations of this section have been worked out in {\sc Mathematica 4.0}, running on a {\em Pentium IV 2.400 Mhz DIMM DDR266 512
MB}.

%We include here some cocyclic Hadamard matrices over $D_{4t}$, for $2 \leq t \leq 13$, which must be understood as the pointwise linear
%combinations of the corresponding 2-cocycles of the basis ${\cal B}$ described in the precedent section.

The table below shows, for every $2\leq t \leq 10$, the number of essays which have been executed looking for cliques in $G_t$, the
average time required in these calculations, the average size of these cliques, the largest size found so far, the time required in this
calculation and one instance among the largest cliques already found.

%\vspace*{-1.4cm}
\begin{center}\begin{table*}[ht]\caption{Results obtained from Algorithm 1.}$$\begin{array}{|c||c|c|c|c|c|l||}\hline
t& \mbox{Essays}& \mbox{Av.Time} &\mbox{Av.Size}  & \mbox{Lar.S.} & \mbox{Time} &\mbox{Clique}  \\
\hline\hline
2 & 10 &  0.0232'' &5 & 5 & 0.016'' & 166, 101, 106, 169, 60 \\
\hline 3 & 10 & 0.039'' & 9 & 9 & 0.031'' & \begin{array}{l} 2396, 730, 881, 940, 1386,2482,1433,\\
2281, 1268\end{array} \\
\hline 4 & 10 & 0.368'' & 13 & 13 & 0.327'' & \begin{array}{l} 22166, 22874, 11698, 34776, 13251,\\ 19428, 49971, 13116, 39594, 43606,\\
38246, 7793, 26218\end{array} \\
\hline 5 & 10 & 0.369'' & 11 & 17 & 0.468'' & \begin{array}{l} 615882, 81761, 124596, 682665,\\ 315178, 323811, 183761, 808837,\\ 405190,
572306, 350936, 677490,\\ 193356, 813428, 219558, 406968,\\ 564460 \end{array} \\ \hline
\end{array}$$\end{table*}\end{center}

\begin{table}[ht]\label{maxcliques}\begin{center}$\begin{array}{|c||c|c|c|c|c|l||}\hline
t& \mbox{Essays}& \mbox{Av.Time} &\mbox{Av.Size}  & \mbox{Lar.S.} & \mbox{Time} &\mbox{Clique}  \\
\hline\hline  6 & 10& 4.128'' & 13.8 & 21 & 3.198'' & \begin{array}{l} 1943217, 2913196,4896610,1866067,\\ 13689735, 5584078, 10118748,
4943241,\\ 3324617, 6759253, 5685752, 8874722,\\ 14710001, 8830356, 6614554,4840148,\\
8771401, 3272036, 3364754,  13712184,\\ 11117018 \end{array} \\
\hline 7 & 10 & 12.19'' & 12.8 & 17 & 20.53'' & \begin{array}{l} 30121113, 153038560, 15673110,\\ 73300362, 41266505, 117929485,\\
86698833, 43440080, 105500358,\\ 176876844, 41595557, 156361628,\\ 172841777, 119125298,
205299228,\\ 79251369, 210612677 \end{array} \\
\hline 8 & 4 & 1'39'' & 13.75 & 15 & 59.37'' & \begin{array}{l} 1305024466, 1295442290, 2999574898,\\ 3795589971, 3631000788, 3248355466, \\
364489541, 2896517604, 2629324346,\\ 1255832164, 504314142, 3570677031, \\
965647459, 1387503024, 1483592489\end{array} \\ \hline  9 & 4 & 13'46'' & 14.25 & 16 & 6'34'' & \begin{array}{l} 56453847575, 43074435680,\\
41936051018, 22138975396,\\ 23156633420, 41987793978,
 \\ 37225385892, 18981806746,\\ 45452914326, 7371644105,\\ 53995320335, 52910835425, \\
60376161688, 19823315036,\\ 10578923147, 28478917930 \end{array} \\ \hline
% \end{array}$$
%$$\begin{array}{|c||c|c|c|c|c|l||}\hline
%t& \mbox{Essays}& \mbox{Av.Time} &\mbox{Av.Size}  & \mbox{Lar.S.} & \mbox{Time} &\mbox{Clique}  \\
%\hline\hline
10 & 2 & 4^h50'6'' & 16 & 16 & 3^h0'50'' & \begin{array}{l} 611551599738, 585720653408,\\ 381375537029, 727566346115,\\
90629862411, 851419790210,\\ 418284730924, 524680378162,\\ 364580931042, 154503136551,\\  747357767564, 199828304181,\\ 99982406086,
319408643342,\\ 422206490833, 34030595756 \end{array} \\ \hline
 \end{array}$\end{center}\end{table}

\vspace*{-.75cm}
As the table below shows, it is remarkable that the sizes of the largest cliques found so far are greater than $\lfloor \frac{4t}{3}
\rfloor -3$ (a third of a full Hadamard matrix of size $4t$) and even $2t-3$ (a half of a full Hadamard matrix of size $4t$), which are
the best asymptotic bounds on the depth of partial Hadamard matrices known so far.

\begin{table*}[ht]\begin{center}\caption[c]{Comparing sizes of cliques.}$\begin{array}{||c|c|c|c|c|c|c|c|c|c||} \hline t&2&3&4&5&6&7&8&9&10 \\ \hline \hline \lfloor \frac{4t}{3} \rfloor -3& -1& 1&2&3&5&6&7&9&10\\
\hline 2t-3 & 1&3&5&7&9&11&13&15&17\\ \hline Lar.S. & 5&9&13&17&21&17&15&16&16 \\ \hline \end{array}$\end{center}\end{table*}

\vspace*{-.5cm}
Although our procedure gives large partial Hadamard matrices (of depth greater than half of a full Hadamard matrix), it is very expensive both in space and time. %from both the space and time points of view.

It would be desirable to find a way to design a procedure running significantly faster and which nevertheless leads to large partial
Hadamard matrices as well.
We describe such an algorithm in the following section.

%\newpage

\section{Heuristic searches}

\label{sec4}

In this section we describe two heuristics for searching for cliques in $G_t$.

The first of them is a straightforward adaptation of the best (as far as we know) genetic algorithm for solving the maximum clique problem (MCP). Since it is very expensive in time, we then describe a second heuristic, which is much faster, in exchange of precision (in the sense that not sufficiently large cliques are obtained).

Nevertheless, it is a remarkable fact that this last procedure admits as input data a clique already constructed. Consequently, one could eventually obtain a larger clique. Experimental results suggest that this actually happens more times than one could initially think (see Table \ref{delauneymejorado}). In fact, it seems to be a good idea to combine this fast search with other more precise procedures.

\subsection{Classical GAs for MCP}

\label{subsec1}

As we commented before, finding the maximum clique of a graph is a NP-Hard problem, and consequently all known exact algorithms for this problem will run in time that grows exponentially with the number of vertices in the graph. This makes these algorithms infeasible even in case of moderately large problem instances. Therefore most of the efforts to solve the maximum clique problem are based on heuristic approaches.

In \cite{SG06} a heuristic based steady-state genetic algorithm for the maximum clique problem is described. The steady-state genetic algorithm generates cliques, which are then extended into maximal cliques by the heuristic. After comparison with the three best evolutionary approaches for the maximum clique problem, they find out that their algorithm outperforms all the three evolutionary approaches in terms of best and average clique sizes found on majority of DIMACS benchmark instances (which are the canonical family of graphs used to test MCP-algorithms).

The main features of the genetic algorithm in \cite{SG06} consist in:

\begin{itemize}

\item {\em Chromosome representation}. A $n$-length bit vector represent a chromosome (i.e. a clique), so that a value of 1 at the ith position indicates that the vertex $i$ is in the clique.

\item {\em Crossover}. They use fitness based crossover, so that a child is constructed bit by bit, receiving each time the bit from one of its parents, with probability proportional to the fitness of the parents. The vector obtained so far may not be a clique, so a repairing function is needed to transform the child into a clique.

\item {\em Repair}. First, all 1 bits corresponding to vertices with significantly low degree (in comparison with the provisional fittest individual) are changed to 0. Then, the repairing procedure introduced by Marchiori in \cite{Mar02} is used, so that repeatedly a vertex of the child is selected at random, and either it or all the vertices not adjacent to it are deleted (i.e. the corresponding bits are fixed to 0), until the child becomes a clique.

\item {\em Mutation}. Mutation consists in simple bit flip mutation, where each bit in the chromosome is flipped with a pre-fixed probability $p_m$. Once again, the repairing function is needed to guarantee that a valid chromosome (i.e. a clique) is obtained.

\item {\em Extension}. Once a valid chromosome (i.e. a clique) is constructed, an extension function is applied, in order to extend the given clique to a maximal clique. The idea is repeatedly adding a vertex with highest degree among the set $S$ of vertices simultaneously adjacent to the vertices already in the clique. When the size of the set $S$ is small enough, then the exhaustive search of Carraghan and Pardalos \cite{CP90} is performed.

\item {\em Fitness}. The fitness of a chromosome is equal to the size of the clique that it represents.

\item {\em Selection}. They use binary tournament selection, where the candidate with better fitness is selected with a pre-fixed probability $p_b$.

\item {\em Replacement policy}. No duplicate chromosomes are permitted in the population. When a new child is constructed, it always replaces the worst member of the population, irrespective of its own fitness.

\end{itemize}

It seems natural trying to adapt this algorithm to our case. Unfortunately, one cannot afford to explicitly construct the adjacency lists of $G_t$, for values $t>5$, since they grow exponentially in $t$, as we showed in the section before.

So the extension function described in \cite{SG06} cannot be applied in the case of the graph $G_t$. Nevertheless, we can use instead the algorithm proposed in the previous section.

\begin{algorithm} GA looking for cliques in $G_t$.

Substitute the extension function  in \cite{SG06} by Algorithm 1.

\end{algorithm}

We show now some executions for $2 \leq t \leq 9$, where the size of the population is fixed to 5 and the maximum number of generations is fixed to 20.

\vspace*{-.4cm}
\begin{table}[ht]\caption{Results obtained from Algorithm 2.}$$\begin{array}{|c||c|c|c|c|c|}\hline \mbox{t}& \mbox{Time} &
\mbox{\# generations} & \mbox{\#\{$C:$ $|C|$ is maximal\} } & |C| & PH \mbox{ size} \\
\hline\hline
 2 & 0,171" & 0 & 5 & 5 & 8 \\ \hline
 3 & 0,359" & 0 & 5 & 9 & 12 \\ \hline
 4 & 1,872" & 0 & 5 & 13 & 16 \\ \hline
 5 & 3,931" & 0 & 3 & 17 & 20 \\ \hline
 6 & 19,36" & 0 & 3 & 21 & 24 \\ \hline
 7 & 4'10" & 20 & 3 & 17 & 20 \\\hline
 8 & 24'37" & 20 & 1 & 21 & 24 \\ \hline
 9 & 4^h 36" & 20 & 1 & 18 & 21 \\ \hline
\end{array}$$\end{table}

We next include explicitly the cliques found so far. For brevity, we give the decimal number representations of the binary $4t$-vectors which form each clique.

\begin{table}[ht]\caption{Cliques obtained from Algorithm 2 listed explicitly.}$$\begin{array}{|c||l|}\hline
 \mbox{t}&  \mbox{Clique} \\ \hline\hline
 2 & 89,169,195,149,101 \\ \hline
 3 &  2396,922,3347,3214,2635,2290,2854,2473,1386 \\ \hline
 4 &  11633,27850,20148,40089,29719,6098,14940,43414, 22981,40038,15011,42693,13740 \\ \hline
 5 & 355028,694485,626266,436558,57188,250510, 647277,308937,189682,\\
   & 809194,832054,113080,211555,381267,603590,684888,584579 \\  \hline
6 &  9162129,15599647,12071353,7947574,2834188,10323501,6722981,7449237,\\
  & 5936794,3845219,2986835,13186247,14739211,5822797,5711193,13280116,\\
  & 14785656,11687509,9747734,11764174,13835043 \\  \hline
7 & 206229172,136297864,145421137,239938137,244421436,78464909, 111023826, \\
&  96903644,49849529,150232156,47868686,213067927,46619057,237397306, \\
&  193760535,174494351,223891867 \\ \hline
8 & 925672665,545242888,131329093,1388111142,3800772080,
2494264851,2084807742,\\
& 3004829231,432983728,1265186469,1437667228,2976213301,
766141292,455842134, \\
   & 3399920157,3380112474,2022530125,
1674167067,3521287529,517486905,1394045560  \\ \hline
9 & 6263028814,17380851458,66723786135,13185706648,61523714894, 11518454374, \\
&
28777615826,31813316858,7484361626,45216625720,10357786275,32845624107,\\
&
64532207782,48511720424,46756144798,24818822896,53243861654,27484036889
\\ \hline
\end{array}$$\end{table}

A comparison with the results obtained in the section before, reveals that the genetic algorithm is not as useful as desired.

On one hand, the only improvements are obtained for the values $t=8,9$, and they are not significantly impressive. Moreover, it seems that quality (in terms of the size of the cliques obtained) relays on the extension function rather than in the genetic procedure itself.

On the other hand, each run of the extension function is very time-consu\-ming, so the required time for executing a full run of the genetic algorithm grows drastically. And there is no a dramatic increase in the size of the obtained cliques in return.

It would be desirable to look for a faster way of extending cliques, attending to the particular properties of our graph $G_t$. We tackle with this question in the next section.

\subsection{A fast heuristic search}

\label{subsec2}

In order to get a faster heuristic search, we need to have a deeper look at the adjacency relations between the vertices in $G_t$.

We first translate Proposition \ref{prop2} (which characterizes the set of $s$-vectors orthogonal to a given $k$-vector in terms of distributions of total coincidences $(\alpha _{i_1},\alpha _{i_2},\alpha _{i_3},\alpha _{i_4})$),  in terms of distributions $(i_1,i_2,i_3,i_4)$ of coincidences in 1s (for the first and last quarters) and 0s (for the second and third quarters).

\begin{proposition}\label{prop3} The set of $s$-vectors orthogonal to a given $k$-vector corresponds to the full set of distributions of $s$-vectors satisfying
tuples of coincidences in 1s (first and fourth quarters) and 0s (second and third quarters) $(i_1,i_2,i_3,i_4)$ such that $i_1+i_2+i_3+i_4=2s+2k-t$. Furthermore, this is possible iff $2s+2k-t \geq 0$. \end{proposition}

\begin{proof}

From Proposition \ref{prop2}, we know that the set of $s$-vectors orthogonal to a given $k$-vector is characterized by those distributions of total coincidences $(\alpha _{i_1}, \alpha _{i_2}, \alpha _{i_3}, \alpha _{i_4})$ such that $\alpha _{i_1}+\alpha _{i_2}+\alpha _{i_3}+ \alpha _{i_4}=2t$.

Since $\alpha _i=t-s-k+2i$ from Lemma \ref{lema3}, the relation above comes to be $$4t-4k-4s+2i_1+2i_2+2i_3+2i_4=2t\Leftrightarrow i_1+i_2+i_3+i_4=2s+2k-t.$$

Now, on one hand, since $0 \leq i_j \leq \min (k,s)$, the  value $i_1+i_2+i_3+i_4$ runs over the range $0 \leq i_1+i_2+i_3+i_4 \leq 4 \min (k,s)$.

On the other hand, since $0 \leq s,k \leq \lfloor \frac{t}{2} \rfloor$, it is clear that $2s+2k-t \leq 2 \min (k,s)$.

Thus, provided $2s+2k-t \geq 0$, this value is in the range valid for $i_1+i_2+i_3+i_4$, and therefore there exists a distribution of coincidences in 1s (at first and fourth quarters, $i_1$ and $i_4$ respectively) and 0s (at second and third quarters, $i_2$ and $i_3$ respectively), such that $i_1+i_2+i_3+i_4=2s+2k-t$.

$\Boxe$

\end{proof}

Now it is apparent that not all possible values $s$ in the range $0 \leq s \leq \lfloor \frac{t}{2} \rfloor$ do provide $s$-vectors orthogonal to a given $k$-vector.

\begin{proposition}
Fixed a $k$-vector {\bf v}, there exist $s$-vectors orthogonal to {\bf v} iff $s \in [ \lceil
\frac{t}{2} \rceil-k,\lfloor \frac{t}{2}\rfloor ]$.
\end{proposition}

\begin{proof}

The upper bound is given in Lemma \ref{lema1}.

On the other hand, from Proposition \ref{prop3}, there exist $s$-vectors orthogonal to a given $k$-vector iff $0\leq 2s+2k-t$. Consequently, $s \geq \frac{t}{2}-k$, that is, $s \geq \D \lceil \frac{t}{2} \rceil-k$.

$\Boxe$

\end{proof}

Furthermore, we may straightforwardly precise the number of $s$-vectors orthogonal to a given $k$-vector, for some fixed $s  \in [ \lceil
\frac{t}{2} \rceil-k,\lfloor \frac{t}{2}\rfloor ]$.

\begin{lemma}
Fixed a valid distribution $(i_1,i_2,i_3,i_4)$, the number of $s$-vectors orthogonal to a given $k$-vector is given by the expression:
$$\left( \begin{array}{l}
k \\
i_1 \\
\end{array}\right)
\left( \begin{array}{l}
t-k \\
s-i_1 \\
\end{array}\right)
\left( \begin{array}{l}
k \\
i_2 \\
\end{array}\right)
\left( \begin{array}{l}
t-k \\
s-i_2 \\
\end{array}\right)
\left( \begin{array}{l}
k \\
i_3 \\
\end{array}\right)
\left( \begin{array}{l}
t-k \\
s-i_3 \\
\end{array}\right)
\left( \begin{array}{l}
k \\
i_4 \\
\end{array}\right)
\left( \begin{array}{l}
t-k \\
s-i_4 \\
\end{array}\right)
$$
\end{lemma}

The following tables show the number of $s$-vectors orthogonal to a given $k$-vector, for $3 \leq t \leq 10$, $0 \leq
k \leq  \lfloor \frac{t}{2} \rfloor$, and $0 \leq s \leq \lfloor \frac{t}{2} \rfloor$.

\begin{table}[ht]\caption{Distribution of $s$-vectors orthogonal to a $k$-vector.}$$\begin{array}{|c||c|c|}\hline
  & \multicolumn{2}{c|}{t=3} \\  \hline\hline
 & k=0 & k=1 \\ \hline
s=0 & 0 &  0  \\ \hline s=1 & 0 & 32  \\ \hline
 \end{array}  \; \begin{array}{|c||c|c|c||c|c|c|}\hline
  & \multicolumn{3}{c||}{t=4} & \multicolumn{3}{c|}{t=5}  \\  \hline\hline
 & k=0 & k=1 & k=2 & k=0 & k=1 & k=2 \\ \hline
s=0 & 0 & 0 & 1 & 0 & 0 &  0 \\ \hline
s=1 & 0 & 81 & 96 & 0 &0  & 216 \\
\hline s=2 & 1296 & 486 & 454 & 0  & 3456  & 3240 \\ \hline
 \end{array}$$

 $$\begin{array}{|c||c|c|c|c||c|c|c|c|}\hline
  & \multicolumn{4}{c||}{t=6} & \multicolumn{4}{c|}{t=7}  \\  \hline\hline
 & k=0 & k=1 & k=2 & k=3 & k=0 & k=1 & k=2 & k=3
 \\ \hline
s=0 & 0 & 0 & 0 & 1 & 0 & 0 & 0 & 0 \\
\hline s=1 & 0 & 0 & 256 & 486 & 0 & 0 & 0 & 768  \\ \hline
s=2 & 0 & 10000 & 14688 & 15795 & 0 & 0 & 40000 & 57024 \\
\hline
s=3 & 160000 & 60000 & 49920 & 47148 & 0 & 480000 & 440000
& 422208 \\ \hline
 \end{array}$$

$$\begin{array}{|c||c|c|c|c|c|}\hline
  & \multicolumn{5}{c|}{t=8}  \\  \hline\hline
 & k=0 & k=1 & k=2 & k=3 & k=4   \\ \hline
s=0 & 0 & 0 & 0 & 0 & 1  \\
\hline
s=1 & 0 & 0 & 0 & 625 & 1536   \\
\hline
s=2 & 0 & 0 & 50625 & 147000 & 183904   \\
\hline
s=3 & 0 & 1500625 & 2352000 & 2601000 & 2655744   \\
\hline
s=4 & 24010000 & 9003750 & 7183750 & 6483750 & 6297030   \\
\hline
 \end{array}$$

 $$\begin{array}{|c||c|c|c|c|c|}\hline
  &  \multicolumn{5}{c|}{t=9}  \\  \hline\hline
  & k=0 & k=1 & k=2 & k=3 & k=4  \\ \hline
s=0  & 0 & 0 & 0 & 0 & 0  \\
\hline
s=1 & 0 & 0 & 0 & 0 & 2000  \\
\hline
s=2  & 0 & 0 & 0 & 243000 & 464000  \\
\hline
s=3  & 0 & 0 & 7203000 & 11210000 & 12912000  \\
\hline
s=4 & 0 & 76832000 & 69629000 & 65367000 & 63430000  \\
\hline
 \end{array}$$

$$\begin{array}{|c||c|c|c|c|c|c|}\hline
  &  \multicolumn{6}{c|}{t=10}  \\  \hline\hline
  & k=0 & k=1 & k=2 & k=3 & k=4 & k=5 \\ \hline
s=0  & 0 & 0 & 0 & 0 & 0 & 1 \\
\hline
s=1 & 0 & 0 & 0 & 0 & 1296 & 3750  \\
\hline
s=2  & 0 & 0 & 0 & 194481 & 858600 & 1200625  \\
\hline
s=3  & 0 & 0 & 9834496 & 32773650 & 48326400 & 53560000  \\
\hline
s=4 & 0 & 252047376 & 407209600 & 453248775 & 468312600 & 472003750  \\
\hline
s=5 & 4032758016 & 1512284256 & 1180754176 & 1041640236 & 978746976 & 960098756  \\
\hline
 \end{array}$$\end{table}

In particular, these results suggest that large cliques in $G_t$ should consist of $k$-vectors, for large values of $k$, close to $\lfloor \frac{t}{2} \rfloor$.

This seems to be so, as the calculations below suggest.

For each $3 \leq t \leq 9$, we choose at random a Hadamard matrix of order $4t$ from Sloane's online library \cite{Slo}, say {\tt had.12}, {\tt had16.4}, {\tt had20.hall.n}, {\tt had24.pal}, {\tt had28.pal2}, {\tt had32.pal}, {\tt had36.pal2}.

We now normalize these matrices, by means of the following algorithm. Notice that since just negation and permutation of columns are used, the Hadamard character of the matrix is preserved.

\begin{algorithm} \label{normalizador} Hadamard normalization

\begin{itemize}

\item Negate those columns consisting of a first negative entry.

\item Now, locate those columns $i$, $1 \leq i \leq 2t$, consisting of a second negative entry. Locate those columns $j$, $2t+1 \leq j \leq 4t$, consisting of a second positive entry. Interchange them.

\item Proceed as the step before, now by quarters. As a result, you will obtain a normalized Hadamard matrix.

\end{itemize}

\end{algorithm}

Now we randomly fix a $\lfloor \frac{t}{2} \rfloor$-vector among the rows of these matrices (for instance, the first such occurrence). The table below shows the distribution of the values $s$, for the $s$-vectors of the remaining rows. In addition, we also include the distribution of (not ordered!) total coincidences $(\alpha _{i_1}, \ldots ,\alpha _{i_4})$.

\begin{table}[ht]\caption{Rows in Hadamard matrices are $k$-rows, for $k \in \{\lfloor \frac{t}{2} \rfloor-1,\lfloor \frac{t}{2} \rfloor\}$.}$$\begin{array}{|c|c|c|c|c|} \hline t & \mbox{row}&  s=\lfloor \frac{t}{2}\rfloor& s=\lfloor \frac{t}{2}\rfloor-1 &  \#\{\mbox{total coincidences}\}\\ \hline 3&4&8&0& (1,1,1,3)\rightarrow 8\\ \hline
4&6&8&4& \begin{array}{c} (2,2,2,2)\rightarrow 4\\(0,2,2,4)\rightarrow 4\\ (1,1,3,3)\rightarrow 4\end{array}\\ \hline
5&5&15&1& \begin{array}{c} (1,3,3,3)\rightarrow 12\\(1,1,3,5)\rightarrow 3\\ (2,2,2,4)\rightarrow 1\end{array}\\ \hline
6&4&11&9& \begin{array}{c} (2,2,4,4)\rightarrow 9\\(2,2,2,6)\rightarrow 1\\ (0,4,4,4)\rightarrow 1\\ (1,3,3,5)\rightarrow 6\\ (3,3,3,3)\rightarrow 3\end{array}\\ \hline
7&5&21&3& \begin{array}{c} (3,3,3,5)\rightarrow 16\\(1,3,5,5)\rightarrow 3\\(1,3,3,7)\rightarrow 2\\ (2,2,4,6)\rightarrow 2\\(2,4,4,4)\rightarrow 1\end{array}\\ \hline
8&4&12&16& \begin{array}{c} (2,4,4,6)\rightarrow 12\\(3,3,5,5)\rightarrow 12\\(1,5,5,5)\rightarrow 2\\(3,3,3,7)\rightarrow 2\end{array}\\ \hline
9&4&26&6& \begin{array}{c} (3,5,5,5)\rightarrow 20\\(1,5,5,7)\rightarrow 6\\(2,4,6,6)\rightarrow 6\end{array}\\ \hline \end{array}$$\end{table}

%\vspace*{-1cm}
The table above suggests that one should focus on $k$-vectors for $k\in \{ \lfloor \frac{t}{2}\rfloor -1, \lfloor \frac{t}{2}\rfloor \}$. Furthermore, the vector of total coincidences per quarter uses to be homogeneously distributed.

With these ideas at hand, we next design a {\em fast} heuristic searching for (or eventually extending) cliques in $G_t$.

\begin{algorithm} \label{alg2} Fast heuristic for extending cliques in $G_t$.

\vspace*{.5cm}

{\tt

\noindent Input: a clique $C$ in  $G_t$\\
\noindent Output: a clique $C'$ in $G_t$ containing $C$

\vspace*{.5cm}
%\noindent $\backslash\backslash$ reproduction starts\\

\noindent $C' = C$\\
\noindent for $k$ from $\lfloor \frac{t}{2} \rfloor$ to $\lfloor \frac{t}{2} \rfloor-1$ step $-1$ do\\
%\noindent $k= \lfloor \frac{t}{2} \rfloor$\\
\hspace*{.5cm} $iter= 0$\\
\hspace*{.5cm} while $iter<t$ do\\
\hspace*{1cm} $\{bool,$v$\}=$buildgrapas$(C',k)$\\
\hspace*{1cm} If $bool$ then $iter = 0$; $C'=C'\cup \{$v$\}$ else $iter =iter+1$\\
\hspace*{.5cm} od\\
%\noindent $k= \lfloor \frac{t}{2} \rfloor-1$\\
%\noindent $iter= 0$\\
%\noindent while $iter<t$ do
%\hspace*{.5cm} $\{bool,$v$\}= buildgrapas(C')$
%\hspace*{.5cm} If $bool$ then $iter = 0$; $C'=C'\cup \{$v$\}$ else $iter =iter+1$\\
\noindent od\\
\noindent $C'$}

\end{algorithm}

The function {\tt buildgrapas} tries to construct an $s$-vector quarter by quarter (in a random ordering), attending to the following aspects:

\begin{itemize}

\item Select a number of total coincidences for the quarter, according to the range valid at this step (i.e. such that equation (\ref{sis1}) can be satisfied), and with probability proportional to the number of its appearances in the set of solutions described in Corollary \ref{bueno}.

\item Once the desired number of total coincidences has been fixed, a genetic procedure is performed, for constructing a valid quarter (i.e. such that (\ref{sis1}) can be satisfied). This heuristic consists of populations of $4t$ individuals. If no valid quarter is found after $4t$ generations, the search ends with a failure. In this case, the last quarter constructed so far is deleted, and the process goes on from this point. This situation is limited to occur at most $t$ times.

\item This search is performed at most 10 times. If no valid vector is constructed after these 10 attempts, the search stops and a {\sc False} boolean is returned.

\end{itemize}

%The table below include some executions, for $3 \leq t \leq 10$.

The table below shows, for every $2\leq t \leq 10$, the number of essays which have been executed looking for cliques in $G_t$, the
average time required in these calculations, the average size of these cliques, the largest size found so far, the time required in this
calculation and one instance among the largest cliques already found.

\vspace*{-.4cm}
\begin{table}[ht]\caption{Results obtained from Algorithm 4.}$$\begin{array}{|c||c|c|c|c|c|l||}\hline
t& \mbox{Essays}& \mbox{Av.Time} &\mbox{Av.Size}  & \mbox{Lar.S.} & \mbox{Time} &\mbox{Clique}  \\
\hline\hline
2 & 10 &  1.4'' &5 & 5 & 1.4'' & 86, 101, 149, 89, 60 \\
\hline 3 & 10 & 0.565'' & 9 & 9 & 0.5'' & \begin{array}{l} 1452, 2396, 1393, 874, 756, 2482, 921,\\
 1242, 2281\end{array} \\
\hline 4 & 10 & 19.603'' & 11 & 12 & 17.28'' & \begin{array}{l} 25542, 15462, 13769, 39626, 50538,\\ 22099, 38294, 22188, 22876, 52421,\\
22947, 27802\end{array} \\
\hline  \end{array}$$ \end{table}

\begin{table}[ht]$$\begin{array}{|c||c|c|c|c|c|l||}\hline
t& \mbox{Essays}& \mbox{Av.Time} &\mbox{Av.Size}  & \mbox{Lar.S.} & \mbox{Time} &\mbox{Clique}  \\
\hline\hline
5 & 10 & 27.369'' & 9 & 17 & 76.924'' & \begin{array}{l} 193425, 586090, 808611, 420073\\ 350674, 408358, 222834, 109987,\\ 315192, 341833, 604856, 662897,\\ 171722, 308468, 121445, 218540,\\ 552900 \end{array} \\
\hline 6 & 10& 51.38'' & 7.6 & 10 & 56.87'' & \begin{array}{l} 13215089, 13015273, 10053835, 12875531,\\ 12954262, 6768979, 3454163, 9730420,\\ 3324617, 6759253, 5685752, 8874722,\\ 999010, 906181 \end{array} \\
\hline
7 & 10 & 1'23'' & 8 & 9 & 1'46'' & \begin{array}{l}159860692, 161016643, 73886147,\\ 142481171, 173143250, 170760901,\\
21953176, 43986594, 203259148 \end{array} \\ \hline
%\hline
8 & 10 & 2'11'' & 7.6 & 9 & 2'53'' & \begin{array}{l} 866473426, 1497876124, 1273533381,\\ 1697995341, 1439971764, 3784746441,\\
2345981979, 2309832360, 242628440\end{array} \\ \hline  9 & 10 & 3'02'' & 7.8 & 9 & 3'05'' & \begin{array}{l} 56095030680, 38163367817,\\
41103334725, 54854789721,\\ 22744355553, 45279148512,
 \\ 4785499937, 52133944916,\\ 60209197830 \end{array} \\ \hline
 \end{array}$$\end{table}

Although the size of the cliques obtained so far are smaller than those constructed by the precedent procedures, it is a remarkable fact that this algorithm is substantially faster. In fact, this procedure should be considered as an extension function for cliques better than a procedure itself for constructing cliques starting from the empty graph.

This idea is supported by the calculations showed in the table below, where cliques $C$ of size $|C|\leq 2t$ in $G_t$ constructed by the procedures described in \cite{dLa00,dLG01} (after normalization by Algorithm \ref{normalizador}) are extended to larger cliques $C'$ (of size $|C'|\geq 2t$, more than a half of a full Hadamard matrix!) with Algorithm \ref{alg2}.

\begin{table}[ht]\label{delauneymejorado}\caption{Algorithm 4 applied to $PH_{m \times 4t}$ in \cite{dLG01} produces $PH_{(m+n)\times 4t}$, with $m+n \geq 2t$.}$$\begin{array}{|c|c|c|c|c|}\hline  t& |C|&|C'|&time& \mbox{added vertices}\\  \hline
4&5 & 12& 17.97''& \begin{array}{l} 21930, 50745, 25500, 13107,\\ 37740, 42330, 51510 \end{array}\\ \hline
5&5 & 7& 19.38''& 118659,341714\\ \hline
6&9 & 15& 1'07''& \begin{array}{l} 3099915, 9123660, 8841105,\\10606050, 4844385, 4692870 \end{array}\\ \hline
7&9 & 9& 1'03''& \\ \hline
8&13& 17& 2'51''&  \begin{array}{l} 1923517785, 3032697660,\\1695521520, 2956808130 \end{array}\\ \hline
\end{array}$$\end{table}

Notice that for $t=7$, the input clique has not been extended to a larger one. We suspect that the input clique is maximal, and therefore a larger clique containing it could not exist.

\section{Conclusions}

In this paper we have described three algorithms looking for pretty large partial Hadamard matrices  (i.e. about a third of a full Hadamard matrix), in terms of cliques of the Hadamard Graph $G_t$.

The first one (Algorithm 1) performed some kind of local exhaustive search, and consequently is expensive from the time consuming point of view. So we decided to design some heuristic for constructing partial Hadamard matrices.

Our first approach (Algorithm 2) consisted in an adaptation of Singh and Gupta's genetic algorithm for the Maximum Clique Problem. Unfortunately, it did not work properly, since it used Algorithm 1 for extending cliques, and consequently was very expensive in time as well.

Algorithm 4 prioritizes the required processing time better than the final size of the partial Hadamard matrix to be obtained. Experimental results show that this algorithm may output pretty large partial Hadamard matrices (larger than half a full Hadamard matrix!), provided a suitable initial clique is given as input data.

All the algorithms that we have presented here are based on the properties of the Hadamard Graph $G_t$ which have been described in Section 2.

It would be an interesting question whether different techniques and methods could be considered for designing alternative algorithms searching for large partial Hadamard matrices. For instance, one could ask about the techniques and methods which have been shown to be useful when manipulating full Hadamard matrices.

Unfortunately, this will not be the case, in general. For instance, consider the case of the cocyclic approach.

More concretely, the cocyclic framework has arised as a promising way to construct (cocyclic) Hadamard matrices
\cite{Hor06,AAFR06,AAFR08,AAFR09}. One could ask whether the cocyclic framework is also a good place to look for partial Hadamard
matrices. Actually, this is not the case.

\begin{proposition}

The depth of any partial Hadamard matrix which is a submatrix of a cocyclic matrix $M_f$ is at most a half of the size of $M_f$.

\end{proposition}

\begin{proof}

Attending to the proof of the cocyclic Hadamard test in \cite{HdL95} (see Lemma 1.4 on p. 281), fixed a multiplicative group
$G=\{g_1=1,\ldots ,g_n\}$ and a cocyclic matrix $M_f=(f(g_i,g_j))$ over $G$, rows $g_i\neq g_k$ in $M_f$ are orthogonal if and only if the
summation of the row $g_ig_k^{-1}(\neq g_1)$ is zero (and consequently, the summation of the row $g_kg_i^{-1}$ as well).

If a row $g_k\neq g_1$ in $M_f$ fails to sum zero, then, for each $1 \leq i \leq n$, the pair of rows $\{g_i,g_k^{-1}g_i\}$ (and also
$\{g_i,g_kg_i\}$) fails to be orthogonal. By partitioning the rows of $M_f$ into pairs of the type $\{g_i,g_k^{-1}g_i\}$, it turns out that
such a pair contributes at most one row to a partial Hadamard matrix included in $M_f$. Consequently, the depth of any partial Hadamard
matrix which is a submatrix of $M_f$ is at most $\frac{n}{2}$, as claimed.

$\Boxe$

\end{proof}

It would be interesting to think about the way in which $k$-vertices in $G_t$ could be combined in order to get larger cliques.

Nevertheless, it would be also interesting to investigate whether improved versions of the algorithms described in this paper could be designed, attending to other considerations.

% For one-column wide figures use
%\begin{figure}
% Use the relevant command to insert your figure file.
% For example, with the graphicx package use
%  \includegraphics{example.eps}
% figure caption is below the figure
%\caption{Please write your figure caption here}
%\label{fig:1}       % Give a unique label
%\end{figure}
%
% For two-column wide figures use
%\begin{figure*}
% Use the relevant command to insert your figure file.
% For example, with the graphicx package use
%  \includegraphics[width=0.75\textwidth]{example.eps}
% figure caption is below the figure
%\caption{Please write your figure caption here}
%\label{fig:2}       % Give a unique label
%\end{figure*}
%
% For tables use
%\begin{table}
% table caption is above the table
%\caption{Please write your table caption here}
%\label{tab:1}       % Give a unique label
% For LaTeX tables use
%\begin{tabular}{lll}
%\hline\noalign{\smallskip}
%first & second & third  \\
%\noalign{\smallskip}\hline\noalign{\smallskip}
%number & number & number \\
%number & number & number \\
%\noalign{\smallskip}\hline
%\end{tabular}
%\end{table}

\begin{acknowledgements}
%We would like to express our gratitude to the EACA2010 Scientific Committee, who invited us to submit our oral contribution at the EACA2010 Conference in order %to be considered for publication in Computing.

All authors are partially supported by the research projects FQM--016 and P07-FQM-02980 from Junta de Andaluc\'{\i}a and MTM2008-06578 from Ministerio de Ciencia e Innovaci\'{o}n (Spain).
\end{acknowledgements}

% BibTeX users please use one of
%\bibliographystyle{spbasic}      % basic style, author-year citations
%\bibliographystyle{spmpsci}      % mathematics and physical sciences
%\bibliographystyle{spphys}       % APS-like style for physics
%\bibliography{}   % name your BibTeX data base

% Non-BibTeX users please use

\end{document}